\documentclass[a4paper]{amsart}
\usepackage[textwidth=15cm,top=3cm, bottom=3cm, hcentering]{geometry}
\usepackage[hidelinks]{hyperref}

\usepackage{amsmath,amscd,amsthm,amssymb,amsfonts}
\usepackage[utf8]{inputenc}
\normalfont
\usepackage[T1]{fontenc}
\usepackage{tikz-cd}
\usepackage{tikz}
\usepackage{enumerate}
\usepackage[all]{xy}
\usepackage{latexsym, verbatim}
\usepackage{graphicx}
\usepackage{color}
\usepackage{url}
\usepackage{lmodern}
\usepackage{slashed}
\usepackage{enumitem}
\usepackage{changepage}
\usepackage{adjustbox}
\usepackage{floatpag}
\usepackage[labelformat=empty]{caption}
\usepackage{soul}
\usepackage{tikz-cd}
\usepackage{tikz}

\usepackage{stackengine}

\newtheorem{theo}{Theorem}[section]
\newtheorem{lemm}[theo]{Lemma}
\newtheorem{prop}[theo]{Proposition}
\newtheorem{cor}[theo]{Corollary}
\theoremstyle{definition}
\newtheorem{defn}[theo]{Definition}
\newtheorem{exam}[theo]{Example}
\newtheorem{defprop}[theo]{Definition-Proposition}

\theoremstyle{remark}
\newtheorem{rmk}[theo]{Remark}

\numberwithin{equation}{section}

\newcommand{\End}{\mathrm{End}}
\newcommand{\Ker}{\mathrm{Ker}}

\newcommand{\Img}{\mathrm{Im}}

\newcommand{\rank}{\mathrm{rank}}

\newcommand{\ov}[1]{\overline{#1}}

\newcommand{\BV}{\mathrm{BV}}
\newcommand{\Hycom}{\mathrm{Hycom}}

\newcommand{\g}{\mathfrak{g}}

\newcommand{\JJ}{\mathbb{J}}
\newcommand{\CC}{\mathbb{C}}

\newcommand{\RR}{\mathbb{R}}

\newcommand{\Z}{\mathbb{Z}}
\newcommand{\R}{\mathbb{R}}

\newcommand{\cA}{{\mathcal A}}
\newcommand{\Aa}{\mathcal{A}}

\newcommand{\Hh}{\mathcal{H}}

\newcommand{\Ll}{\mathcal{L}}

\newcommand{\Tt}{\mathcal{T}}

\newcommand{\del}{\partial}
\newcommand{\delb}{{\bar \partial}}
\newcommand{\mub}{{\bar \mu}}

\newcommand{\lab}{{\bar \lambda}}
\newcommand{\la}{\lambda}
\newcommand{\La}{\Lambda}
\newcommand{\taub}{{\bar \tau}}

\newcommand{\dR}{\mathrm{dR}}

\newcommand{\Joana}[1]{{\color{blue}{#1}}}
\newcommand{\Scott}[1]{{\color{red}{#1}}}

\title{Homotopy BV-algebras in Hermitian geometry}

\author[J. Cirici]{Joana Cirici}
\address[J. Cirici]{Departament de Matemàtiques i Informàtica, Universitat de Barcelona\\
Gran Via 585\\
08007 Barcelona, Spain}
\email{jcirici@ub.edu}

\author[S. Wilson]{Scott O. Wilson}
  \address[S. Wilson]{Department of Mathematics, Queens College, City University of New York, 65-30 Kissena Blvd., Flushing, NY 11367}
  \email{scott.wilson@qc.cuny.edu}

  \thanks{
J. Cirici acknowledges the Serra H\'{u}nter Program, the Spanish State Research Agency (EUR2023-143450, PID2020-117971GB-C22 and CEX2020-001084-M), the French National Research Agency (ANR-20-CE40-0016 HighAGT) and
Govern de Catalunya (2021-SGR-00697). S. O. Wilson acknowledges support provided by the Simons Foundation Award program for Mathematicians (\#963783) and a PSC-CUNY Award, jointly funded by The Professional Staff Congress and The City University of New York (TRADB  \# 66358-00 54). 
}

\begin{document}

\begin{abstract}
We show that the de Rham complex of any almost Hermitian manifold carries a natural commutative $\BV_\infty$-algebra structure satisfying the degeneration property. In the almost K\"ahler case, this recovers Koszul's BV-algebra, defined for any Poisson manifold. As a consequence, both the Dolbeault and the de Rham cohomologies of any compact Hermitian manifold are canonically endowed with homotopy hypercommutative algebra structures, also known as formal homotopy Frobenius manifolds. 
Similar results are developed for (almost) symplectic manifolds with Lagrangian subbundles.
\end{abstract}

\maketitle

\section{Introduction}

The notion of Batalin–Vilkovisky algebra plays an important role in geometry,
topology and mathematical physics.
Initially developed in the context of quantum field theory as a tool to
compute path integrals in the presence of symmetry, the BV formalism  was subsequently translated to the context of supermanifolds equipped with symplectic or Poisson structures.
In algebraic topology, BV-algebras have proven to be extremely useful, with applications to
string topology, the study of cohomological operations, and to deformation quantization, amongst others.

Motivated by the mirror symmetry program, Barannikov and Kontsevich \cite{BaKo} gave a recipe which, out of a BV algebra satisfying a certain degeneration property, one obtains a hypercommutative algebra structure in cohomology.
Hypercommutative algebras are governed by the homology operad of the Deligne-Mumford-Knudsen compactification of the moduli space of genus zero surfaces with marked points, and lead to formal Frobenious manifold structures, which in turn are related to quantum cohomology.
The approach via deformation theory of Barannikov and Kontsevich has been recently refined and generalized
by means of operadic machinnery \cite{KMS}, \cite{DCV}, \cite{DSV}. In its higher homotopical generalization, any  $\BV_\infty$-algebra satisfying the so-called degeneration property gives rise to a homotopy hypercommutative algebra structure.

Various geometries are known to encode BV-algebra structures satisfying the degeneration property: Koszul \cite{Koszul} showed that the de Rham algebra of any Poisson manifold
carries a BV-operator, defined using the exterior differential and the contraction operator by the Poisson bi-vector field. This in particular applies to symplectic manifolds. The degeneration property was noticed by Dotsenko, Shadrin and Vallette \cite{DSV}, who also extended the theory to Jacobi manifolds and their associated commutative $\BV_\infty$-algebras. Similarly, Braun and Lazarev \cite{BrLa} prove that generalized Poisson structures (given by multivector fields) give rise to commutative $\BV_\infty$-algebras satisfying degeneration.
Here we enlarge this list to including Hermitian, and almost Hermitian manifolds. 

A useful result that will guide all our geometric examples is the fact that, given a commutative dg-algebra together with an operator $\Lambda$ of order $\leq 2$,
one always obtains a commutative $\BV_\infty$-algebra by conjugating the differential with $\exp(\Lambda \xi)$, where $\xi$ is a formal parameter (see Proposition \ref{prop;sillychoice}). The underlying multicomplex satisfies the degeneration property, by construction. The $\BV_\infty$-operators are in this case given by 
\[\Delta_k={1\over k!}\underbrace{[\Lambda,[\Lambda,\cdots[\Lambda}_{k},d]]].\]
For Hermitian manifolds, the main observation is that a suitable generalization of the K\"ahler identities, first due to Demailly in the Hermitian case \cite{De}, and derived independently by
Tomassini and Wang \cite{TW} and  
Fernandez and Hosmer in the almost Hermitian case \cite{FH}, allows for an explicit description of the commutative $\BV_\infty$-algebra obtained after conjugation by $\exp(\Lambda \xi)$, where $\Lambda$ is the formal adjoint to the Lefschetz operator. In short,  for any Hermitian manifold, the tuple
\[
(\Aa_\CC, \wedge, \, \Delta_0=\delb \, , \, \Delta_1=-i( \del^* + \tau^*) \, ,\, \Delta_2= i\la^* \,)
\]
is a commutative $\BV_\infty$-algebra extending the Dolbeault algebra (see Theorem~\ref{mainHermitian}). Here $(\Aa_\CC,\wedge)$ is the complex de Rham algebra of the manifold, $\Delta_2$ is the adjoint of wedging with the $3$-form $\la = \del \omega$, where $\omega$ is the fundamental form associated to the Hermitian metric, and $\tau :=  [\La, \la]$ is the
 torsion operator. Note that both $\tau$ and $\lambda$ vanish in the K\"ahler case. 
 
 Similarly, there is a real de Rham commutative $\BV_\infty$-algebra for any almost Hermitian manifold, 
 \[
\left( \Aa, \wedge, d, \, \Delta_1 =  - (d_c^* + T_c^*) \, , \, \Delta_2 = d \omega_c^* \, \right)
 \]
 where $d_c^* :=-\star J^{-1} d J\star$ is the formal adjoint of $d_c:=J^{-1}dJ$, $d \omega_c^*$ is the adjoint of $d\omega_c = J^{-1} d\omega J$, and  $T_c^*$ is the adjoint of $T_c:=J^{-1}[\Lambda , d\omega]J$.
 This generalizes Koszul's $\BV$-algebra on the de Rham forms of an almost K\"ahler (and hence symplectic) manifold, since in that case $d \omega =0$ and  $T:=[\Lambda , d\omega]=0$ (see Theorem~\ref{mainAH}). 
Using canoncial deformation retracts given by Hodge theory one obtains, in the compact case, canonical hypercommutative algebra structures that are non-trivial in general.

By similar methods, an additional geometric set-up we consider is that of (almost) symplectic manifolds equipped with a Lagrangian subbundle. Indeed, there is real-de Rham version, as well as a ``Dolbeault-type'' version in the case that the Lagrangian subbundle is integrable (Theorem \ref{thm;s-L}). Such Lagrangian-type structures have garnered significant attention recently  in mathematical considerations of  Type II-A data \cite{LTY},  \cite{HK}, \cite{FPPZ}.
It is natural to wonder if the  commutative $\BV_\infty$-algebras obtained here, and its induced homotopy hypercommutative algebras, play a role in a theory of mirror symmetry, beyond the K\"ahler realm.

\medskip 

Let us briefly explain the structure of this paper. Section \ref{algprel} includes algebraic preliminaries on $\BV$ and hypercommutative algebras. Hermitian manifolds are treated in  Section \ref{HermitianSec}. We then treat the almost Hermitian setting in Section  \ref{SecAHAS}, where we also study (almost) symplectic manifolds with integrable polarizations. Lastly, Section \ref{SectionExamples} is devoted to examples, using Chevalley-Eilenberg algebras in nilmanifold theory. We have tried to make this paper accessible to both geometers and homotopy theorists, keeping the operadic machinery to a minimum and recalling the main geometric structures needed, in each section. In this sense, we intend the last section on examples to be quite pedagogical.

\subsection*{Acknowledgments}
The authors would like to thank Vladimir Dotsenko, Geoffroy Horel, Fernando Muro and Bruno Vallette for many useful discussions.

\section{Algebraic preliminaries}\label{algprel}

In this purely algebraic section, we collect the main results on $\BV$-algebras and their relation to hypercommutative algebras needed later. Most of these appear in or are inspired by the works of Khoroshkin, Shadrin and Markarian \cite{KMS} and Dotsenko, Shadrin and Vallette \cite{DSV}. We will use cohomological grading.  For the reader that is mainly interested in the applications to complex and symplectic geometry, yielding (homotopy) $\BV$-algebras, the more operadic methods of this section might be skipped until it is clear they are necessary, in particular for development of homotopy hypercommutative algebras.

\subsection{Multicomplexes and the gauge equation} We review how any multicomplex satisfying the degeneration property
determines an explicit solution to the gauge equation in terms of the data given by any deformation retract.
\begin{defn} \label{defn;multicplx}
A \emph{multicomplex} $(A, \Delta_0= d, \Delta_1, \Delta_2 ,\ldots)$ is a $\Z$-graded vector space $A$ with a family of 
linear operators $\Delta_i$ of degree $1-2i$ satisfying
\[
\sum_{i= 0}^n \Delta_i \Delta_{n-i} = 0 \quad \text{for all }\quad n \geq 0.
\]
\end{defn}

In particular, $(A,d)$ is a cochain complex, and we denote its cohomology by $H(A)$.
Any multicomplex with $\Delta_i = 0$ for $i \geq 2$, is a double complex with $|d| = +1$ and $| \Delta_1| = -1$. 
For a multicomplex of the form $(A, \Delta_0= d, \Delta_1, \Delta_2 ,\ldots, \Delta_n)$,
the defining equation is equivalent to \[(d + \Delta_1 + \cdots +\Delta_n)^2=0.\]

Multicomplex structures may be transferred to cohomology using deformation retracts as follows:
\begin{defn}
A \textit{deformation retract} of a cochain complex $(A,d)$ is given by a diagram
\[
\begin{tikzcd}[ampersand replacement = \&]
(A,d) \arrow[r, rightarrow, shift left, "\rho"] \arrow[r, shift right, leftarrow, "\iota" swap] \arrow[loop left, "h"] \& (H(A),0),
\end{tikzcd}
\]
where $\iota$ and $\rho$ are morphisms of complexes
such that 
$\rho \iota=1$
and $dh+hd=\iota\rho-1$.
\end{defn}
Given a multicomplex $(A, \Delta_0= d, \Delta_1, \Delta_2 ,\ldots)$, a choice of a deformation retract $(\iota,\rho,h)$ for $(A,d)$
 defines a multicomplex structure $(H(A),\Delta_0'=0,\Delta_1',\Delta_2' ,\ldots)$ on cohomology, where 
\[\Delta_n':=\sum_{\substack{j_1+\cdots+j_k=n\\ j_i\geq 1}} \rho \Delta_{j_1}(h\Delta_{j_2})\cdots(h\Delta_{j_k})\iota\text{ for all }n\geq 1.\]

An important special class of multicomplexes arises when imposing the so-called degeneration property, for which $\Delta_i'=0$. This was introduced in \cite{Terilla}, \cite{KKP} for $\BV$-algebras but, in fact, the degeneration property only affects the underlying double complex. A generalization to multicomplexes appears in \cite{BrLa} and \cite{DSV}.
We refer to the latter reference for the various equivalences in the definition below.

\begin{defprop}
 A multicomplex $(A, \Delta_0= d, \Delta_1, \Delta_2 ,\ldots)$ is said to have the \textit{degeneration property} if the following equivalent conditions are satisfied:
 \begin{enumerate}
   \item There exists a solution   \[\varphi(\xi)=\sum_{i\geq 1} \varphi_i \xi^i\in \mathrm{End}(A)[[\xi]]\] 
to the \textit{gauge equation}
\[e^{\varphi(\xi)} d e^{-\varphi(\xi)}=d+\Delta_1 \xi+\Delta_2 \xi^2+\cdots.\] 
  \item There is a deformation retract $(\iota,\rho,h)$ for $(A,d)$ such that $\Delta_i'=0$ on $H(A)$, for all $i\geq 1$.
  \item Any deformation retract $(\iota,\rho,h)$ for $(A,d)$ satisfies $\Delta_i'=0$ on $H(A)$, for all $i\geq 1$.
  \item The spectral sequence associated to the multicomplex degenerates at the first stage.
 \end{enumerate}
\end{defprop}

An expression for a solution $\varphi(\xi)$ to the gauge equation in terms of a deformation retract $(\rho,\iota,h)$ is given in \cite{DSV} for the case of bicomplexes.
Following the same strategy, we give below a formula in the case of multicomplexes. In order to not unnerve the reader, we point out that only the first-order term's explicit formula,
\[
\varphi_1 = h \Delta_1 - \iota \rho \Delta_1 h,
\]
is needed in the remainder of the paper, when we will compute examples of hypercommutative operations on nilmanifolds.

\begin{prop}  \label{prop: defretrTOpsi}
Let $(A,d,\Delta_1,\Delta_2,\ldots)$ be a multicomplex satisfying the degeneration property and let  $(\iota,\rho,h)$ be a deformation retract.
A solution to the gauge equation is given by
\[\varphi(\xi)=-\log\left(1-
\sum_{n\geq 1}
h\Delta_n\xi^n
+
\sum_{\substack{n\geq 1\\j_1+\cdots+j_k=n\\ j_i\geq 1}}(\iota  \rho \Delta_{j_1} h\Delta_{j_2}h\cdots \Delta_{j_k}h)\xi^n
\right).\]
\end{prop} 

\begin{proof}A solution to the gauge equation of the form
\[\varphi(\xi)=-\log(1+\sum_{n\geq 1}f_n\xi^n)\]  is equivalent to a collection of maps
$f=\{f_n:A\to A\}$ for $n\geq 1$
satisfying
\[[d,f_n]=\sum_{i+j=n; i,j>0}f_{i}\Delta_j.\]
In turn, this is equivalent to having an $\infty$-isotopy of multicomplexes 
$f:(A,d,\Delta_i)\rightsquigarrow (A,d,0)$ (so $f$ is an $\infty$-morphism with $f_0=1$).
By homotopy transfer theory of multicomplexes (see for instance Proposition 1.4 of \cite{DSV}), the deformation retract $(\iota,\rho,h)$ gives an $\infty$-morphism
$\rho_\infty:(A,d,\Delta_i)\rightsquigarrow (H(A),0,0)$ by letting
\[\rho_0=\rho\text{ and }\rho_n=\sum_{\substack{j_1+\cdots+j_k=n\\ j_i\geq 1}} 
\rho \Delta_{j_1} h\Delta_{j_2}h\cdots \Delta_{j_k} h\text{ for }n\geq 1.\]
Also, we have an $\infty$-quasi-isomorphism $q_\infty:(A,d,\Delta_i)\rightsquigarrow (A,d,0)$ given by
\[q_0=1-\iota\rho\text{ and }q_n=-h\Delta_n\text{ for }n\geq 1.\]
We then obtain the required $\infty$-isotopy by letting $f=\iota \rho_\infty+q_\infty$.
\end{proof}

\subsection{Batalin-Vilkovisky and hypercommutative algebras}
We first recall the definition of the algebraic order of linear operators on a graded commutative unital algebra $(A,\wedge)$, due to Grothendieck \cite{Gro}. Let $\mathcal{D}_0^k$ be the space of $A$-linear endomorphisms of $A$ of degree $k$, which is isomorphic to $A^k$ as the space of (left) multiplication operators, $a \mapsto L_a$.
Inductively, the degree $k$ operators of algebraic order $\leq r$,  denoted $\mathcal{D}_r^k$, is the space of linear homomorphisms $f$ of degree $k$ such that $[L_a ,f] \in \mathcal{D}_{r-1}^{k+j}$ for all $L_a \in \mathcal{D}_0^j$.  The bracket always indicates the graded-commutator, 
\[[f_1 ,f_2] = f_1 f_2 - (-1)^{|f_1|\cdot |f_2|} f_2 f_1.\] 
%We note that derivations of $A$ are all order $1$, though not all order $1$ operators are derivations. 

Using the fact that $(\End(A), \circ, [\,,\,])$ is a graded Poisson algebra, i.e. the bracket satisfies the Jacobi identity and $[f_1,-]$ is a derivation of composition, it is immediate to show that
\[
\mathcal{D}_r^k \cdot \mathcal{D}_s^\ell \subset \mathcal{D}_{r+s}^{k+\ell},
\quad \quad \quad 
[\mathcal{D}_r^k , \mathcal{D}_s^\ell ] \subset \mathcal{D}_{r+s-1}^{k+\ell}.
\]

\begin{defn} \label{commBVinf}  A \textit{commutative $\BV_\infty$-algebra} is a tuple $(A, \wedge, d, \Delta_1, \Delta_2 ,\ldots)$ such that
$(A,\wedge,d)$ is a commutative dg-algebra,  $(A, d, \Delta_1, \Delta_2 ,\ldots)$ is a multicomplex, and each $\Delta_n$ has  order $\leq n+1$. 
%A commutative $\BV_\infty$-algebra with $\Delta_n =$ for $n \geq 2$ is called a (commutative) dgBV-algebra.
\end{defn}

Note that (differential graded) $\BV$-algebras are precisely those commutative $\BV_\infty$-algebras such that $\Delta_n=0$ for all $n\geq 2$. 

Let $c\BV_\infty$ denote the operad governing commutative $\BV_\infty$-algebras. The obvious map of operads  $c\BV_\infty\to \BV$ is a quasi-isomorphism. One may view $c\BV_\infty$ as a partial cofibrant resolution of $\BV$, where only the operation $\Delta$ is resolved (but not the commutative structure). Algebras over a cofibrant (resp. minimal) resolution of $\BV$ are studied in \cite{GTV} and \cite{DCV}, respectively, and are called $\BV_\infty$ (resp. skeletal $\BV_\infty$) algebras.
For our purposes, it will suffice to work with the more restrictive notion of commutative $\BV_\infty$-algebra introduced above.

The last ingredient required to state the main result of this section are hypercommutative algebras. These are defined as algebras over the operad 
\[\Hycom(n):= H_*(\overline{\mathcal{M}}_{0,n+1}),\] the Deligne-Mumford-Knudsen
moduli space of stable genus 0 curves. The following algebraic description appears in \cite{Getzler}. 

\begin{defn}
A \textit{hypercommutative algebra} is a cochain complex $(A,d)$ with a sequence of
graded symmetric $n$-ary operations
\[m_n:A^{\otimes n}\to A\] of degree $2(2-n)$, compatible with the differential and
satisfying the following generalized associativity condition:
for all $n\geq 0$ and $a,b,c,x_j\in A$, 
\[\sum_{S_1\cup S_2=\{1,\cdots,n\}} \pm m_*(m_*(a,b,x_{S_1}),c,x_{S_2})=
\sum_{S_1\cup S_2=\{1,\cdots,n\}} \pm m_*(a,m_*(b,c,x_{S_1}),x_{S_2}),
\]
where $\pm$ is the Koszul sign rule and $x_S$ denotes $x_{s_1},\cdots,x_{s_m}$ for 
$S=\{s_1,\cdots,s_m\}$.
\end{defn}

The operad $\Hycom$ is Koszul, with Koszul dual
cooperad $H^{*+1}({\mathcal{M}}_{0,*+1})$, so that the cobar construction yields a cofibrant resolution
\[\Omega H^{*+1}({\mathcal{M}}_{0,*+1}) \to H_*(\overline{\mathcal{M}}_{0,*+1}).\]
By definition then, a homotopy hypercommutative algebra is an algebra over the operad
\[\Hycom_\infty:=\Omega H^{*+1}({\mathcal{M}}_{0,*+1}).\]

The following result can be found in Proposition 10 of \cite{DSV3} (see also Proposition 4.4 of \cite{DSV}).
The proofs therein use homotopy transfer techniques. We give below an alternative proof which follows from a quasi-isomorphism of operads $\Hycom\to \BV/\Delta$, given by Khoroshkin, Markarian and Shadrin in \cite{KMS}. Here
$\BV/\Delta$ is a model of the homotopy quotient of $\BV$ by $\Delta$ and will be explained in the proof.

\begin{theo}\label{hycomprop}
Let $(A, \wedge, d, \Delta_1, \Delta_2 ,\ldots)$ be a commutative $\BV_\infty$-algebra. Any solution 
 \[
 \varphi(\xi)=\sum_{i\geq 1} \varphi_i \xi^i\in \mathrm{End}(A)[[\xi]]
 \] 
 to the  gauge equation 
\[
e^{\varphi(\xi)} d e^{-\varphi(\xi)}=d+\Delta_1 \xi+\Delta_2 \xi^2+\cdots
\]
defines a unique homotopy hypercommutative algebra structure on $(A,d)$ extending the commutative product and such that its homotopy type determines the homotopy type of the initial commutative $\BV_\infty$-algebra.
\end{theo}
\begin{proof}
A commutative $\BV_\infty$-algebra equipped with the trivialization data $\varphi(\xi)$
is a representation of the homotopy quotient of $c\BV_\infty$ by the operators $\Delta_i$, which we denote by 
$c\BV_\infty/\{\Delta_i\}_{i>0}$. 
A model for this homotopy quotient may be defined as follows: we add, to $c\BV_\infty$, a family of generators $\{\varphi_i\}$ of arity 1 and degree $-2i$. The differential $\delta$ of $\varphi_i$ is defined by writing  $\varphi(\xi)=\sum \varphi_i \xi^i$ and imposing the gauge equation.
This gives the various terms
\[\arraycolsep=4pt\def\arraystretch{1.4}
\begin{array}{l}
\Delta_1=[d,\varphi_1],\\
\Delta_2=[d,\varphi_2]+{1\over 2}[[d,\varphi_1],\varphi_1],\\
\Delta_3=[d,\varphi_3]+[[d,\varphi_1],\varphi_2]+[[d,\varphi_2],\varphi_1]+{1\over 6}[[[d,\varphi_1],\varphi_1],\varphi_1],\\
\cdots
\end{array}
\]
We then let $\delta \varphi_i:=[d,\varphi_i]$. The lower terms are:
\[
\arraycolsep=4pt\def\arraystretch{1.4}
\begin{array}{l}
\delta \varphi_1= \Delta_1,\\
\delta\varphi_2=-{1\over 2}[\Delta_1,\varphi_1]+\Delta_2,\\
\delta\varphi_3=-[\Delta_1,\varphi_2]-[\Delta_2,\varphi_1]+{1\over 3}[[\Delta_1,\varphi_1],\varphi_1]+\Delta_3,\\
\cdots
\end{array}
\]
The map of operads $c\BV_\infty\to \BV$ defined by sending $\Delta_1\mapsto \Delta$ and $\Delta_i\mapsto 0$ for $i>1$ descends to a surjective quasi-isomorphism, $\pi: c\BV_\infty/\{\Delta_i\}_{i>0} \to \BV/\Delta$, of the respective homotopy quotients, where the homotopy quotient $\BV/\Delta$ is defined exactly as above but with $\Delta_i=0$ for $i>1$ (see also \cite{GuMu}). By 
\cite{KMS}, there is a quasi-isomorphism of operads $\theta:\Hycom\to \BV/\Delta$.
Therefore, by the lifting property, the dotted arrow in the diagram
\[
\xymatrix{
&&c\BV_\infty/\{\Delta_i\}_{i>0}\ar@{->>}[d]^\pi\\
\Hycom_\infty\ar[r]\ar@{.>}[rru]^{\theta'}&\Hycom\ar[r]^{\theta}&\BV/\Delta
}
\]
exists and the diagram strictly commutes. Moreover, all maps are quasi-isomorphisms.

Finally, the composition 
\[\BV_\infty \to c\BV_\infty\to c\BV_\infty/\{\Delta_i\}_{i>0}\]
has a lift 
\[
\xymatrix{
&\Hycom_\infty\ar[d]^{\theta'}\\
\BV_\infty\ar@{.>}[ur]\ar[r]&c\BV_\infty/\{\Delta_i\}_{i>0},
}
\]
which makes the diagram commute up to homotopy.
\end{proof}

As is customary, a $c\BV_\infty/\{\Delta_i\}$-algebra is an algebra over the operad $c\BV_\infty/\{\Delta_i\}$.

The following will be useful to make explicit computations on geometric examples.
It follows by combining Proposition \ref{prop: defretrTOpsi} and Theorem \ref{hycomprop}.

\begin{cor}  \label{cor;cBVdefretrTOhh}
Let $(A, \wedge, d, \Delta_1, \Delta_2 ,\ldots)$ be a commutative $\BV_\infty$-algebra whose underlying multicomplex satisfies the degeneration property. Then any deformation retract $(\rho,\iota,h)$ for $(A,d)$
determines a unique  homotopy hypercommutative algebra on $A$. 
The strict arity-3 hypercommutative product induced on $A$ is given by
\[
\arraycolsep=4pt\def\arraystretch{1.4}
\begin{array}{ll}
m_3(a,b,c)=&
\varphi_1(a\wedge b\wedge c)+\varphi_1(a)\wedge b\wedge c +a\wedge \varphi_1(b)\wedge c +a\wedge b\wedge \varphi_1(c)\\
&-\varphi_1(a\wedge b)\wedge c-(-1)^{|b|\cdot|c|}\varphi_1(a\wedge c)\wedge b-a\wedge \varphi_1(b\wedge c).
\end{array}
\]
where $\varphi_1=h\Delta_1-\iota\rho \Delta_1 h$.
\end{cor}

Note that the operation $m_3$ is a strict hypercommutative operation, rather than a higher homotopy operation. In particular, whenever $m_3$ is non-zero in cohomology, the homotopy hypercommutative structure induced on $H(A)$ is not $\infty$-isomorphic to a homotopy commutative algebra. On the other hand, an 
initial commutative $\BV_\infty$-algebra may be $\infty$-isomorphic to a homotopy commutative algebra, or even a commutative dg-algebra, as in the geometric cases below, but still admit non-trivial hypercommutative structures.
We state a general fact to this effect here.

\begin{prop} \label{prop;sillychoice}
Let $(A,\wedge,d)$ be a commutative dg-algebra and suppose $\Lambda\in\End(A)$  
is an operator of degree $-2$ and order $\leq 2$. For all $k\geq 1$, let 
\[\Delta_k={1\over k!}\underbrace{[\Lambda,[\Lambda,\cdots[\Lambda}_{k},d]]].\]
Then:
\begin{enumerate}
\item The tuple $(A,\wedge, d,\Delta_i)$ is a commutative $\BV_\infty$-algebra,
which is quasi-isomorphic to the commutative dg-algebra $(A,\wedge,d)$ as commutative $\BV_\infty$-algebras.
\item The choice of gauge solution $\varphi(\xi)=\Lambda \xi$ yields a homotopy hypercommutative algebra which is quasi-isomorphic to its underlying homotopy commutative algebra.
\end{enumerate}
\end{prop}

\begin{proof}
The fact that the $(A, d,\Delta_i)$ is a multicomplex follows from the identity
\[\mathbb{D}:=e^{\Lambda\xi}de^{-\Lambda \xi}=Ad_{e^{\La \xi}}(d) = e^{ad_{{\La \xi }}} (d)=\sum_{k\geq 0}{1\over k!} (ad_{\La \xi })^k d=\sum_{k\geq 0}{[\Lambda,-]^k\xi^k\over k!}  d=d+\sum_{k\geq 1} \Delta_k\xi^k.\]
Indeed, $\mathbb{D}^2=0$ is the multicomplex condition.
Since $d$ has order $\leq 1$ and $\Lambda$ has order $\leq 2$, one inductively shows that $\Delta_i$ has order $\leq i+1$, using that the commutator of operators lowers the order by at least 1. This proves that $(A,\wedge, d,\Delta_i)$ is a commutative $\BV_\infty$-algebra.

Consider now the commutative dg-algebra $(A(t,dt),\wedge,d)$ and define 
$\Lambda_t:=\Lambda\otimes t \in\End(A(t,dt))$, so that
\[\Lambda_t(\sum a_it^i+b_it^i dt):=\sum \Lambda(a_i)t^{i+1}+\Lambda(b_i)t^{i+1}dt.\]
This is again an operator of degree $-2$ and order $\leq 2$ and so the same argument as before 
shows that $(A(t,dt),\wedge,d,\Delta(t)_i)$ is a 
commutative $\BV_\infty$-algebra 
with 
\[\Delta(t)_k={1\over k!}\underbrace{[\Lambda_t,[\Lambda_t,\cdots[\Lambda_t}_{k},d]]].\]
The evaluation map $\delta^\lambda:A(t,dt)\to A$ given by $t\mapsto \lambda$ and $dt\mapsto 0$ satisfies 
$\delta^1\Delta(t)_k=\Delta_k$ and $\delta^0\Delta(t)_k=0$ for all $k\geq 1$.
Therefore we obtain quasi-isomorphisms of commutative $\BV_\infty$-algebras
\[(A,\wedge,d,\Delta_i)\stackrel{\delta^1}{\longleftarrow} (A(t,dt),\wedge,d,\Delta(t)_i)\stackrel{\delta^0}{\longrightarrow} (A,\wedge,d,0).\]
This proves the first claim.

The gauge solution $\varphi(\xi)=\Lambda_t \xi$ gives a $\BV_\infty/\{\Delta_i\}$-algebra structure 
on $A(t,dt)$ with $\varphi_1=\Lambda_t$ and $\varphi_i=0$ for $i>0$. Since $\delta^k\Lambda_t=\Lambda_k\delta^k$, the above string of quasi-isomorphisms promotes to a string of quasi-isomorphisms of  $\BV_\infty/\{\Delta_i\}$-algebras. Therefore the corresponding homotopy hypercommutative algebras are also 
quasi-isomorphic.
\end{proof}

 \begin{rmk}
The above result recovers in particular the hypercommutative triviality results of Guan and Muro \cite{GuMu} for Poisson and Jacobi manifolds.
In the geometric examples below, the first case of Proposition \ref{prop;sillychoice} will always occur, and so we will not choose the trivial gauge solution mentioned in the second case.
Rather, in the geometric situations throughout this paper, the gauge solutions will be obtained from deformation retracts that are induced canonically by the underlying structure, e.g. the choice of metric.
\end{rmk}

\section{Hermitian manifolds}\label{HermitianSec}
Let $(M,J, g)$ be a Hermitian $2n$-manifold (not necessarily compact) with fundamental $(1,1)$-form $\omega(X,Y)= g(JX,Y)$.  We will denote by $\Aa=\Aa_{\dR}(M)$ its de Rham algebra and by $\Aa_\CC=\Aa\otimes_\RR\CC$ its complexification. 
The almost complex structure $J:TM\to TM$ induces a bi-degree decomposition by forms of type $(p,q)$
\[
\Aa^k_\CC=\bigoplus_{p+q=k} \Aa^{p,q},
\]
and the exterior differential $d$ decomposes into two components $d=\del+\delb$, where $\del$ has bidegree $(1,0)$ and $\delb$ is its complex conjugate, of bidegree $(0,1)$.

Let $d^*:=-\star d \star$ be the formal adjoint of $d$ with respect the Hodge-star operator $\star: \cA^{k}_\CC \to \cA^{2n-k}_\CC$
which satisfies $\star^2= (-1)^k$ on $\cA^k$.
Consider the real differential operator 
\[d_c:=J^{-1} \circ d \circ J=i(\delb-\del)\]
and denote
 by $d_c^*:=-\star d_c \star$ its formal adjoint. It
satisfies \[d_c^*=i(\del^*-\delb^*),\]
where $\del^*:= -\star \delb \star$ and $\delb^*:=-\star \del \star$ are the formal adjoints to $\del$ and $\delb$, respectively.

Recall there is the $\mathfrak{sl}_2$-representation $\{L,\La, H = [L,\La]\}$ on $\cA$, where $L \eta := \omega \wedge \eta$ is the Lefschetz operator,
$\La:=L^* = \star^{-1} L \star$ is the metric adjoint of $L$, and $H$ acts on $\cA^k$ by $(k-n)$, where $n$ is the complex dimension of $M$. We will also consider the operators
 \[\la := [\del, L]=\del w\quad\text{ and }\quad\tau := [\La, \la]\]
 of bidegrees $(2,1)$ and $(1,0)$ respectively, as well as their adjoints
 $\la^*$ of bidgree $(-2,-1)$ and  $\tau^*$ of bidegree $(-1,0)$.

For a K\"ahler manifold we have $\la = \tau = 0$, and the tuples
\[ 
 (\Aa_\CC, \wedge, \Delta_0=\delb, \Delta_1=-i \del^*)
 \quad \text{ and }\quad
(\Aa, \wedge, \Delta_0=d, \Delta_1=-d_c^*)
\] 
are commutative differential graded $\BV$-algebras over $\CC$ and $\RR$, respectively (see \cite{Koszul} and \cite{CaZh}), and both satisfy the degeneration property (in fact, they satisfy the stronger so-called $\Delta_0\Delta_1$-condition).
Note,  in both cases, the fundamental K\"ahler identity \[[\La,\delb] = -i \del^*\] is needed for the bicomplex relation $\Delta_0\Delta_1+\Delta_1 \Delta_0=0$ to hold.

 For Hermitian manifolds,
 the operators $\delb$ and $\del^*$ do not anti-commute, but the  fundamental Hermitian identity 
  \[
[\La, \delb] = -i(\del^* + \tau^*),
\]
 due to Demailly \cite{De}, generalizes the K\"ahler  identity above,
 and is proved by a local calculation. We remark that Demailly used this relation to prove several 
 vanishing theorems for Hermitian vector bundles \cite{De2}, and more recently it established a Lefschetz duality for Hermitian manifolds \cite{SWH}. With this identity at hand, we have:
 
\begin{theo}\label{mainHermitian}
For any Hermitian manifold the tuple
\[
(\Aa_\CC, \wedge, \, \Delta_0=\delb \, , \, \Delta_1=-i( \del^* + \tau^*) \, ,\, \Delta_2= i\la^* \,)
\]
is a commutative $\BV_\infty$-algebra. The total differential 
\[
\mathbb{D} =  \delb -i \left(\del^* + \tau^*\right)\xi  + i\la^*\xi^2 \in \End(\Aa_\CC)[[\xi]]
\]
 satisfies
\[
\mathbb{D} =  e^{\La \xi} \delb e^{- \La \xi}.
\]
\end{theo}

\begin{proof}

By (1) of Proposition \ref{prop;sillychoice} it suffices to show that 
\[
\Delta_1= -i(\del^*+\tau^*), \quad
 \Delta_2 = i\lambda^*, \quad
 \Delta_3=0.
\]
The first equation is Demailly's identity, $[\Lambda,\delb]=-i(\del^*+\tau^*)$, and for the second we compute
 \[
 \Delta_2={1\over 2}[\Lambda,[\Lambda,\delb]]=-{i\over 2}([\Lambda,\del^*]+[\Lambda,\tau^*]).
 \]
First, one checks using $\del^*=-\star \delb \star$, that $[\Lambda,\del^*]=[\del,L]^*=\lambda^*$. 
%\Scott{ proof: $[\Lambda,\del^*] =  ( \star^{-1} L \star)(-\star \delb \star) - (-\star \delb \star) (\star^{-1} L \star) =  - \star^{-1} L (-1)^{2n-k+1} \delb \star) + \star \delb L \star) = - (-1)^{2n-k+1-2} \star L (-1)^{2n-k+1} \delb \star) + \star \delb L \star) =\star [\delb ,L]\star = \star \lab \star = +\la^*
On the other hand, since $\tau=[\Lambda,\lambda]$, we have
\[
[\Lambda,\tau^*]=[\Lambda,[\lambda^*,L]]=[\lambda^*,[\Lambda,L]]= [H, \lambda^*] = -3\lambda^*.
\]
To explain, in the second equality we used the Jacobi identity, together with the fact that $[\Lambda,\lambda^*]=
[\lambda,L]^*=0$, since the graded commutator of wedging with $\del \omega$ and $\omega$ vanishes. Since $\la^*$ has degree $-3$, we have $[H, \lambda^*] = -3\lambda^*$. Altogether, this implies $\Delta_2 =i\lambda^*$. Finally, $\Delta_3=[\Lambda, i\lambda^*] =0$, as already noted.
\end{proof}

There is also a real commutative $\BV_\infty$-algebra with underlying differential $d$, simply by adding the previous structure to its complex conjugate. 
Indeed, since $\La$ is a real operator, we can define the real operator $T := [\La, d\omega \wedge -]$, which decomposes as
$T = \tau + \taub$, while $[d,L] = d\omega = \la + \lab$. Conjugating by $J$ we have
\[
T_c =i(\taub-\tau) \quad \quad d \omega_c = i(\lab^*-\la^*),
\]
and
\[
T_c^*=i(\tau^*-\taub^*)\quad\quad d \omega_c^* =i(\la^*-\lab^*).
\]
Recall $d_c:=J^{-1}dJ$ and $d_c^* := - \star d_c \star$.
We then have:

\begin{theo}\label{mainHermitianReal}
For any Hermitian manifold
the tuple
\[(\Aa, \wedge, \, \Delta_0=d \, , \, \Delta_1=-(d_c^* + T_c^*) \, , \, 
\Delta_2=d \omega_c^*) \, \]
is a commutative $\BV_\infty$-algebra. The total differential 
\[
\mathbb{D} =  d -i \left(d_c^* + T_c^*\right)\xi  + d\omega_c^* \, \xi^2 \in \End(\Aa_\CC)[[\xi]]
\]
satisfies 
\[
\mathbb{D} =  e^{\La \xi} \delb e^{- \La \xi}.
\]
\end{theo}

Theorems \ref{mainHermitian} and \ref{mainHermitianReal} imply
 that the underlying multicomplexes satisfy the degeneration property.
In order to obtain non-trivial and canonical homotopy hypercommutative structures, we will use the following deformation retracts, which are functorial for maps of Hermitian manifolds.

By Hodge theory any de Rham (resp. Dolbeault) cohomology class admits a unique $d$-harmonic (resp. $\delb$-harmonic) representative.
Moreover, there are projection maps
\[\pi_d:\Aa^{k}\to \Hh_d^{k}\text{ and } \pi_\delb:\Aa_\CC^{p,q}\to \Hh_\delb^{p,q}\]
to $d$-harmonic and 
 $\delb$-harmonic forms respectively, as well as Green operators
\[G_d:\Aa^{k}\to \Aa^{k}\text{ and }G_\delb:\Aa_\CC^{p,q}\to \Aa_\CC^{p,q}\]
defined by zero on harmonic forms and by the inverse of the corresponding Laplacian
on their orthogonal complement. We recall these deformation retractions here, which were also used in \cite{CiHoBV}).  We 
note that the choice of sign here in defining $h$ is to simplify the formulas in Proposition~\ref{prop: defretrTOpsi}.

\begin{lemm}[De Rham deformation retract]\label{defcanonicalR}
For a compact Riemannian manifold there is a canonical deformation retract $(\iota,\rho,h)$
for the de Rham complex $(\Aa,d)$ where $\iota$ and $\rho$ are given by taking $d$-harmonic representatives, and projecting to $d$-harmonic forms, respectively, and \[h:=-d^* G_d.\]
Explicitly, using the isomorphism $d: \Img \, d^* \to \Img \, d$, we have \[h = -d^{-1} \circ \pi_{|\Img d}: \Aa \to {\Img \, d^*}.\]
 \end{lemm}
  
\begin{lemm}[Dolbeault deformation retract]\label{defcanonicalC}
For a compact Hermitian manifold there is a canonical deformation retract $(\iota,\rho,h)$
for the de Dolbeault complex $(\Aa,\delb)$ where the maps 
 $\iota$ and $\rho$ are given by taking $\delb$-harmonic representatives, and projecting to $\delb$-harmonic forms, respectively, and \[h:=-\delb^* G_\delb = -\delb^{-1} \circ \pi_{|\Img \delb}: \Aa_\CC \to {\Img \, \delb^*}.\]
\end{lemm}

 Using these canonical deformation retracts, we obtain:

\begin{cor}\label{hycomHermitian}
 The de Rham and Dolbeault cohomologies of any compact Hermitian manifold carry canonical homotopy hypercommutative algebra structures extending the commutative product. 
 Such homotopy hypercommutative algebras 
 encode the homotopy type of the initial commutative $\BV_\infty$-algebras of 
  Theorems \ref{mainHermitian} and \ref{mainHermitianReal}, respectively.
\end{cor}

\begin{proof}
By Theorems \ref{mainHermitian} and \ref{mainHermitianReal} respectively we have commutative $\BV_\infty$-algebras on the Dolbeault (resp. de Rham) algebra of forms, and both satisfy the degeneration property.
The result follows from  Theorem \ref{hycomprop} together the canonical deformation retracts of Lemmas \ref{defcanonicalC} and \ref{defcanonicalR}, respectively.
\end{proof}

Note that even if the Frölicher spectral sequence of a complex manifold degenerates at the first stage, the Dolbeault and de Rham cohomologies of the manifold are not in general isomorphic as commutative algebras.
In particular, in general, the above two homotopy hypercommutative algebra structures will not be homotopy equivalent. In contrast, in the case of K\"ahler manifolds, in \cite{CiHoBV} it is shown that the Dolbeault and de Rham algebras are in fact quasi-isomorphic as $\BV$-algebras, and the associated homotopy hypercommutative algebra structures are in this case formal. We will give examples of non-formal and non-trivial hypercommutative structures on Hermitian manifolds in Section \ref{SectionExamples}.

\section{Almost Hermitian and almost symplectic manifolds}\label{SecAHAS}
In this section we extend the previous results to the case of almost Hermitian and almost symplectic manifolds. In the particular case of almost K\"ahler (and thus symplectic) manifolds, one recovers the $\BV$ and hypercommutative structures defined by \cite{Koszul} and \cite{DSV}, defined for any Poisson manifold. 
Therefore, the results may be viewed as a natural extension of the theory of hypercommutative structures for Poisson manifolds to (almost) Hermitian manifolds.
 Finally, the (almost) symplectic case is shown to have bigraded refinement, given the data of an integrable Lagrangian subbundle of the real tangent bundle. We view this symplectic-Lagrangian version, on the real differential forms, as the symplectic analogue of the Hermitian case of the previous section.

\subsection{Almost Hermitian manifolds}
With small modifications, the results of Section \ref{HermitianSec} generalize easily to the case of almost Hermitian manifolds, so we will be brief. This requires a generalization of the K\"ahler identities which, to our knowledge, was first deduced by Tomassini and Wang \cite{TW}. These identities were independently derived, via a different method using algebraic calculations in the Clifford algebra, by Fernandez and Hosmer \cite{FH}, and we follow the notation given therein. The fundamental relation is
\[
[\La, d]= - (d_c^* + T_c^*) 
\]
where 
\[
T= [\La, [d,L]] =[\La, d \omega \wedge -].
\]
where $d\omega  \wedge - $ is the operator of degree $3$ given by wedging with $d\omega$, we let \[
T_c= J^{-1} T J\text{ and }
     d\omega_c=J^{-1} d\omega J,
     \]
and $T_c^*$ and $d\omega_c^*$ are the adjoints.

An argument similar to the proof of Theorem \ref{mainHermitian} yields:

\begin{comment}
 One can compute that
\begin{align*}
d_c & = i(\delb-\del-\mub+\mu) \\
d_c^* &=  i(\del^* -\delb^*+\mub^* -\mu^*).
\end{align*}
Similarly, we can decompose $T := [\La, d\omega \wedge - ]= \tau_\delb + \tau_\del + \tau_\mub + \tau_\mu$
with  $| \tau_\mub | =(-1,1)$, $|\tau_\delb| = (0,1)$, $|\tau_\del|= (0,1)$ and $|\tau_\mu| = (1,-1)$,
we can compute
\begin{align*}
T_c & = i ( \tau_\delb - \tau_\del - \tau_\mub + \tau_\mu) \\
T_c^* &= i( \tau_\del^* -  \tau_\delb^*  + \tau_\mub^* - \tau_\mu^*) .
\end{align*}
Finally, decomposing  $d\omega = \la_\delb + \la_\del + \la_\mub + \la_\mu$ 
with 
$| \la_\mub | =(0,3)$, $|\la_\delb| = (1,2)$, $|\la_\del|= (2,1)$ and $|\la_\mu| = (3,0)$,
we have
\begin{align*}
d\omega_c &= i( \la_\delb - \la_\del - \la_\mub + \la_\mu) \\
d \omega_c^* &= i(  \la_\del^* - \la_\delb^* - \la_\mu^* +   \la_\mub^* ).
\end{align*}
\end{comment}

\begin{theo}\label{mainAH}
For any almost Hermitian manifold, the tuple
\[
\left( \Aa, \wedge, d, \, \Delta_1 =  - (d_c^* + T_c^*) \, , \, \Delta_2 = d \omega_c^* \, \right)
\]
is a commutative $\BV_\infty$-algebra on the real differential forms. The total differential 
\[
\mathbb{D} =  d +  \Delta_1 \xi + \Delta_2 \xi^2 \in \End(\Aa)[[\xi]]
\]
 satisfies
\[
\mathbb{D} =  e^{ \La \xi} d e^{- \La \xi}.
\]
\end{theo}
\begin{proof}
 By (1) of Proposition \ref{prop;sillychoice} it suffices to show that 
\[\Delta_1=[\Lambda,d]=-(d_c^* + T_c^*),\, \Delta_2={1\over 2}[\Lambda,[\Lambda,d]]=d \omega_c^*\text{ and }\Delta_3={1\over 6}[\Lambda,[\Lambda,[\Lambda,d]]]=0.\]
The expression for $\Delta_1$ is cited above. Next, note that 
\[
[\Lambda,[\Lambda,d]]=-[\Lambda,d_c^*]-[\Lambda,T_c^*].\]
One checks using $d^* = - \star d \star$ that $[\Lambda,d_c^*]=d\omega_c^*$.
%\Scott{proof is $[\La, d_c^*] = J^{-1} [\La,d^*]J = J^{-1} (\star^{-1} L \star (-\star d \star) - (- \star d \star  \star^{-1} L \star   )J = J^{-1}(-1)^{2n-k+1+2} ((-1)^{2n-k+1+1} \star L d \star) - (- \star d L \star   )J = \star J^{-1} [d,L] J \star = d\omega^*$}
 Next,  we have 
\[
[\Lambda, T_c^*]= J^{-1}[\La,T^*]J =  J^{-1}[T,L]^*J = -3 d\omega_c^*
\]
where we used
\[
[T,L] = [[\La,d\omega],L] = -[H,d\omega] = - 3 d\omega,
\]
using the Jacobi identity.  This gives $\Delta_2 = d \omega_c^*$ as desired.
Finally $\Delta_3 =0$ since $[\La, d\omega_c^*] = [d \omega_c , L ]^* = 0$.
\end{proof}

Note in the integrable case, the relation $[\La, d]= - (d_c^* + T_c^*)$ reduces to the fundamental identity of Demailly for Hermitian manifolds, so we recover Theorem  \ref{mainHermitianReal}.
Moreover, the above result, together with the canonical deformation retract of Lemma \ref{defcanonicalR}
allow to generalize Corollary \ref{hycomHermitian} to almost Hermitian manifolds. We obtain:

\begin{cor}\label{hycomAH}
 The de Rham cohomology of any compact almost Hermitian manifold carries a canonical homotopy hypercommutative algebra structure extending the commutative product. 
\end{cor}

As in the case of complex manifolds, the complex de Rham algebra 
of any almost Hermitian manifold admits a bidegree decomposition induced by the $\pm i$-eigenspaces of $J$.
In this general case, the exterior differential decomposes as 
\[ d=  \mub + \delb + \del + \mu\]
with bidegrees given by
$|\mub| =(-1,2)$, $|\delb| = (0,1)$, $|\delb|= (0,1)$ and $|\mu| = (2,-1)$. Note that
$\mu$ and $\mub$ vanish if and only if $J$ is integrable. 
However, there is no obvious bigraded version of Corollary \ref{hycomAH}.
While Dolbeault cohomology may still be defined as in \cite{CWDol},
there is no Hodge theory available in general,
and so we don't have canonical deformation retracts for Dolbeault cohomology in the non-integrable case.
As we will later see, a bigraded version may be defined for the case of symplectic manifolds admitting real polarizations.

\subsection{Almost symplectic and almost K\"ahler manifolds}

In this subsection we indicate an almost-symplectic version and observe that whenever the fundamental form is closed, $d \omega = 0$, the induced commutative $\BV_\infty$-algebra of Theorem \ref{mainAH} is the $\BV$-algebra induced by Koszul \cite{Koszul}, for symplectic manifolds. In particular, it is metric independent.  Therefore, the induced hypercommutative structures agree as well.

Recall from Koszul \cite{Koszul}, on a Poisson manifold $(M,\pi)$ there is a 
differential graded $\BV$-algebra 
\[(\Aa,\wedge, d,\Delta=[i_\pi,d]),\]
where $i_\pi:\Aa^*\to \Aa^{*-2}$ is defined by contracting forms 
with the Poisson bi-vector field.
Moreover, the gauge equation
\[e^{i_\pi\xi} de^{-i_\pi \xi}=d+\Delta \xi\]
is satisfied. In particular, this fits into the setting of Proposition \ref{prop;sillychoice}.
Recall, the choice of solution $\varphi(\xi)=i_\pi \xi$ to the gauge equation leads to trivial hypercommutative structures, as first noted in \cite{GuMu}. However, by Theorem \ref{hycomprop} and choosing a Riemannian metric, the canonical deformation retract for the de Rham algebra of Lemma \ref{defcanonicalR} gives homotopy hypercommutative
structures for Poisson manifolds, which are non-trivial in general (as shown in Example \ref{KTPoisson} below).

Recall that given a symplectic manifold $(M,\omega)$, there is a symplectic Hodge-star operator 
 \[
 \star_\omega: \cA^k \to \cA^{2n-k}\quad\text{ 
 defined by } \quad
\alpha \wedge \star_\omega \beta = \langle\alpha, \beta
\rangle_\omega \frac{\omega^n}{n!},
\]
where $\langle -,- \rangle_\omega$ is the pairing on $\cA^*$ induced by $\omega$.
Explicitly, if $\alpha = a_1 \wedge \cdots \wedge a_k$ and $\beta = b_1 \wedge \cdots \wedge b_k$ then $\langle  \alpha, \beta \rangle_\omega = \det ( \omega^{-1}(a_i,b_j))$. One can verify that $\star_\omega^2=1$, and $\star_\omega (1) = \frac{\omega^n}{n!}$ is the volume form,
and define the pairing
\[
(\alpha,\beta)_\omega = \int_M \langle  \alpha, \beta \rangle_\omega  \frac{\omega^n}{n!}
\]
in the case that $M$ is closed.

Letting $L\eta = \omega \wedge \eta$, the symplectic adoint of $L$, denoted by $\Lambda = L^\omega$, is defined by 
$\langle  \Lambda \alpha, \beta \rangle_\omega = \langle  \alpha, L \beta \rangle_\omega$, and satisfies
$\Lambda= \star_\omega L \star_\omega$. More generally, we can similarly define the symplectic adjoints of any operators $\phi_1$ and $\phi_2$ of algebraic order zero, which satisfies $[\phi_1,\phi_2]^\omega = [\phi_2^\omega,\phi_1^\omega]$.
The symplectic adjoint of such a zeroth-order operator $\phi$, of degree 
$|\phi|$, equals $\phi^\omega = (-1)^{|\phi|(k-\phi)} \star_\omega \phi \star_\omega$ on $\cA^k$. 
\begin{comment}
\Scott{ proof: For $\beta \in \cA^k$ and $\alpha \in \cA^{k - |\phi|}$,
\begin{align*}
 \langle\phi \alpha,\beta  \rangle_\omega =
\star_\omega( \phi \alpha \wedge \star_\omega \beta)
  &= (-1)^{|\phi| (k - |\phi|)} \star_\omega ( \alpha \wedge \phi \star_\omega \beta) \\ 
&= (-1)^{|\phi| (k - |\phi|)} \star_\omega ( \alpha \wedge  \star_\omega(  \star_\omega \phi \star_\omega \beta))  \\
&= \langle\alpha,  (-1)^{|\phi| (k - |\phi|)}  \star_\omega \phi \star_\omega \beta  \rangle_\omega
\end{align*}
This does NOT hold for first order operators on compact manifolds.
} 
\end{comment}
In particular, for the degree $3$ operator $[d,L]$, of algebraic order zero, we have
\[
[d,L]^\omega = (-1)^{k+1} \star_\omega [d,L] \star_\omega \quad \textrm{on $\cA^k$}.
\]

The symplectic adjoint of $d$ is defined by 
\[
d^\omega := (-1)^k \star_\omega d \star_\omega \quad \textrm{on $\cA^k$},
\]
and satisfies $(d\alpha,\beta)_\omega = (\alpha,d^\omega \beta)_\omega$ when $M$ is compact. 
\begin{comment}
\Scott{proof: for $\beta \in \cA^k$ so $\alpha \in \cA^{k-1}$, using Stokes' Theorem 
\[
(d\alpha,\beta)_\omega = \int   d \alpha \wedge \star_\omega \beta 
 = -(-1)^{k-1} \int \alpha \wedge d \star_\omega \beta 
 =(-1)^{k} \int \alpha \wedge  \star_\omega^2 d \star_\omega \beta
 = (\alpha, (-1)^k \star_\omega d \star_\omega \beta)_\omega
\]
}
\end{comment}

To connect with the Poisson case, a non-degenerate Poisson bi-vector field $\pi$ defines a contraction operator $i_\pi$, and in the symplectic case, $i_\pi = -\La$ (c.f. \cite{Bry} p. 96) and the relation $[i_\pi,d]= \Delta$ of Koszul \cite{Koszul} becomes
\[
[\La,d]= d^\omega. 
\]
 We can now state an almost-symplectic analogue.

\begin{theo} \label{thm;AS}
Let $(M,\omega)$ be an almost symplectic manifold. The tuple
\[(\Aa, \wedge, d , \Delta_1 :=d^\omega + [\La, [d,L]]^\omega, \Delta_2 = -[d,L]^\omega )
\]
is a commutative $\BV_\infty$-algebra.
The total differential
\[
\mathbb{D} =  d + (d^\omega +  [\La, [d,L]]^\omega) \xi -  ([d,L]^\omega)  \xi^2 \in \End(\Aa)[[\xi]]
\]
 satisfies
\[
\mathbb{D} = e^{\Lambda \xi} d e^{ -\Lambda \xi}.
\]
\end{theo}

\begin{proof}
 In \cite{TW}, Tomassini and Wang show that, for any almost symplectic manifold,
\[
[d,\La] = d^\Lambda + [[\La, d^\La],L],
\]
where the sign convention $d^\La := (-1)^{k+1} \star_\omega d \star_\omega$ on $\cA^k$ is used,
so that $d^\omega = - d^\La$ . We claim this is equivalent to 
\[
[\La, d] = d^\omega + [\La, [d,L]]^\omega. 
\]
 To justify the right-hand summand we compute 
\[
 [[\La,d^\omega],L]=  [[d,L]^\omega,L] = [\La, [d,L]]^\omega,
\]
where the first equality is justified by the formal adjoint calculation
\[
[\La,d^\omega] = \La (-1)^k \star_\omega d \star_\omega - (-1)^k \star_\omega d \star_\omega \La 
=
(-1)^k \star_\omega [L,d] \star_\omega = [d,L]^\omega.
\]

The remainder of the proof is similar to the previous cases. From \cite{TW}, we have
 \[
 \Delta_1 = [\La, d] = d^\omega + [\La, [d,L]]^\omega.
 \]
Then computing we have
\[
2 \Delta_2 = [\La,  d^\omega + [\La, [d,L]]^\omega ] = [\La,  d^\omega] + [\La, [\La, [d,L]]^\omega ]
\]
where $[\La,  d^\omega] = [d,L]^\omega$ and 
\[
[\La, [\La, [d,L]]^\omega ] = - [L, [\La, [d,L]]]^\omega = - [H, [d,L]]^\omega = -3 [d,L]^\omega,
\]
so that $\Delta_2 = -[d,L]^\omega$. Finally, $\Delta_k = 0$ for $k \geq 0$ since
$[\La, [d,L]^\omega] = [L, [d,L]]^\omega = 0$.
\end{proof}

In the symplectic case, with $[d,L]=d\omega=0$, this recovers Koszul's $\BV$-algebra as 
\[(\Aa,\wedge,d,\Delta=d^\omega).\]

Now, suppose we have an almost Hermitian manifold, i.e. an almost symplectic structure $\omega$, with compatible almost complex structure $J$, so that $g(X,Y) = \omega(X,JY)$ is a Riemannian metric. Then $g$ determines a metric  $\langle -,- \rangle$ on each space $\cA^k$, and the (Riemannian) Hodge star is 
 \[
 \star: \cA^k \to \cA^{2n-k}\quad\text{  defined by } \quad 
 \alpha \wedge \star \beta = \langle\alpha, \beta \rangle \Omega.
\]
where $\Omega$ is a volume form. Recall that $\star^{2}=(-1)^k$ on $\cA^k$, and $\star = \star_\omega J = J \star_\omega$, where $J$ is the extension to $\cA^*$ as an algebra automorphism. 

 \begin{lemm} For any almost Hermitian manifold $(M, g, J, \omega)$ we have
\begin{align*}
 d^*_c &= - d^\omega \\
 d \omega_c^* & = -[d,L]^\omega \\
 T_c^* & =  - [\La, [d,L]]^\omega.
 \end{align*}
 \end{lemm}
 
\begin{proof} First, the metric adjoint of a zeroth order operator $\phi$ of degree $|\phi|$ equals to 
\[
\phi^* = (-1)^{|\phi|(k-\phi)} \star^{-1} \phi \star \quad \textrm{on $\cA^k$.}
\]  
In particular, 
$[d,L]^*= (-1)^{k+1} \star^{-1} [d,L] \star$ and $d\omega_c^* = (-1)^{k+1} \star^{-1} d\omega_c \star$ on $\cA^k$.
Then, using the fact that $J \star = \star J = (-1)^k \star_\omega$ on $\cA^k$, we compute on $\cA^k$ that
\begin{align*}  
d^*_c = - \star J^{-1} d J \star &= -(-1)^{k+1} \star^{-1} J^{-1} d J \star \\
& = (-1)^{k} (J \star)^{-1} d (J \star)  \\
& = (-1)^k (-1)^{k+1} \star_\omega d   (-1)^k  \star_\omega = - d^\omega,
\end{align*}
and
\begin{align*}
d \omega_c^* = (-1)^{k+1} \star^{-1} J^{-1} d \omega J \star   
&=   (-1)^{k+1} (-1)^{k+3} \star_\omega^{-1}  d \omega  (-1)^k \star_\omega \\ 
& = (-1)^k  \star_\omega [d,L]  \star_\omega = -[d,L]^\omega,
\end{align*}
and
\[
T_c^* = [\La,d\omega_c]^*= [d \omega_c^*,L] = -[[d,L]^\omega,L] =-[\La,[d,L]]^\omega.
\]\qedhere
\end{proof}

\begin{cor}
The commutative $BV_\infty$-structures of Theorem \ref{mainAH} and Theorem \ref{thm;AS} coincide for almost Hermitian manifolds, and in particular depend only on $\omega$, not $J$ and the metric individually.
In the almost K\"ahler case, $d \omega = 0$, the commutative $\BV_\infty$-algebra of Theorem \ref{mainAH} is
\[
\left( \Aa, \wedge, d, \, \Delta_1 =  - (d_c^* + T_c^*) \, , \, \Delta_2 = d \omega_c^* \, \right)
=
\left( \Aa, \wedge, d, \, \Delta_1 = d^\omega  \,, \Delta_2 = 0\, \right),
\]
which is Koszul's $\BV$-algebra.
\end{cor}

\subsection{Symplectic manifolds with real polarizations}
An almost complex manifold $(M,J)$ has a complex splitting, or polarization,
$T_\CC M = T^{0,1}M \oplus T^{0,1} M$ inducing a decomposition of differential forms. In the integrable case, this leads to the bigraded commutative $\BV_\infty$ and hypercommutative structures on Dolbeault cohomology, as we have seen in the previous sections. 
In the case of an (almost) symplectic manifold $(M,\omega)$ of dimension $2n$, we can similarly ask for a splitting of the real tangent bundle into subbundles
$TM = \Ll \oplus \Ll'$
where $\Ll$ and $\Ll'$ are Lagrangian, i.e. $\rank \,\Ll = n$ and $\omega|_\Ll = \omega|_{\Ll'} = 0$. Sometimes this is referred to a \textit{real polarization}. 

Note that for any Lagrangian sub-bundle $\Ll$, and any almost complex structure $J$ on $M$ compatible with $\omega$, we have that $\Ll'=J\Ll$ is also Lagrangian, and $\Ll$ are $\Ll'$ are orthogonal. So, in particular such polarizations exist for every Lagrangian subbundle $\Ll$. The prototypical non-compact example is the cotangent bundle of a manifold, with tautological symplectic form.

A real polarization $(\Ll,\Ll')$ for the tangent bundle of a symplectic manifold $(M,\omega)$ induces a decomposition of de Rham differential forms
\[
\cA^k = \bigoplus_{p+q=k} \cA^{p,q}\quad\text{ with } \quad \cA^{p,q} = \Gamma(M, \wedge^p (\Ll')^* \otimes \wedge^q \Ll^*).
\]
This is an orthogonal direct sum decomposition with respect to the induced pairing 
\[\langle  \alpha, \beta \rangle = \det ( \omega^{-1}(a_i,b_j))\] on $\cA^k$. 
%Note that this pairing is symmetric if $k$ is even, and skew if $k$ is odd.

With respect to this decomposition, the symplectic form $\omega$ has bidegree $(1,1)$, since
$\omega  = \sum_{i=1}^n  e^i \wedge{f^i}$, where $\{e^i\}$ is some  basis of $\Ll^*$, and $\{f^i\}$  of $(\Ll')^*$. Additionally, 
 $\star_\omega :\cA^{p,q} \to \cA^{n-q,n-p}$. 

As in the case of almost complex manifolds, the exterior differential decomposes into $4$ components,
which by analogy we denote by
\[d=\mub+\delb+\del+\mu,\]
where the bidegrees of each component are given by 
\[
|\mub|=(-1,2), \,  |\delb|=(0,1), \, |\del|=(1,0), \text{ and } \,  |\mu|=(2,-1).
\]
To prove this, it suffices to note that 
the algebra of forms is generated by degrees $0$ and $1$, the claim holds in these degrees, and $d$ is a derivation. In particular, there is a spectral sequence associated to this multicomplex decomposition,
which is the real-symplectic analogue of the Fr\"ohlicher-type spectral sequence in \cite{CWDol} defined for almost complex manifolds, and which converges to de Rham cohomology.

For any almost symplectic manifold we have symplectic adjoints $\delta^\omega$, for any of $\delta = \mub,\delb,\del,\mu$. One can check from $d^\omega= (-1)^k \star_\omega d \star_\omega$ on $\cA^k$, and the bidegrees, that
\[
\del^\omega= (-1)^k \star_\omega \delb \star_\omega,
\]
and similarly for other $\delta$.

A Lagrangian subbundle $\Ll$ is said to be \textit{integrable} if it is closed under the Lie bracket of vector fields. We have:

\begin{lemm} Let $(M,\omega)$ be an almost symplectic manifold with real polarization $(\Ll,\Ll')$. 
\begin{enumerate}
 \item The Lagrangian sub-bundle $\Ll$ is integrable if and only if $\mub=0$. In this case, $\delb^2 = 0$ and so
$(\cA^{*,*}, \delb)$ is a cochain complex. 
 \item If both $\Ll$ and $\Ll'$ are integrable then $\mub=\mu=0$ and  
 $(\cA^{*,*} , \delb , \del)$
is a double complex. 
\end{enumerate}
\end{lemm}
\begin{proof}
Using Cartan's formula, relating $d$ and the Lie bracket,  we compute for any $(1,0)$-form $\eta$
\[
(\mub \eta) (X,Y) = (\pi^{0,2}  \circ d \eta) (X,Y) =  d\eta (\pi_{(0,2)} (X,Y)) = - \eta ( [ \pi_{0,1} X,  \pi_{0,1} Y]),
\]
which vanishes for all $X,Y$ if and only if  $\Ll$ is closed under Lie bracket.  In this case, since $d^2=0$ always implies  $\delb^2 + \mub\del+\del\mub=0$, we obtain $\delb^2 = 0$. The second statement follows similarly.
\end{proof}

We emphasize that only the integrability of $\Ll$ is required (not $\Ll'$) to define the complex  $(\cA^{*,*} , \delb)$. This occurs in many natural examples (see for instance Example \ref{ex;Iwas-L} below). It is an open problem in some cases to determine if an integrable Lagrangian has an integrable Lagrangian complement (see \cite{HK}).

Note that if $\Ll$ is integrable, then it integrates to the Lagrangian submanifold $L\subset M$ with tangent bundle $\Ll$, and 
we have $(\cA(L) , d) \cong \left( \bigwedge^* \Ll^* ,\delb \right)$ as commutative dg-algebras, so in particular $H_\delb \left( \bigwedge^* \Ll^*\right) \cong H_d(L)$.

The following is the symplectic-Lagrangian analogue of the (integrable) Dolbeault case:

\begin{theo} \label{thm;s-L}
Let $(M,\omega)$ be a symplectic manifold with a real polarization $(\Ll,\Ll')$. Assume $\Ll$ is integrable.
The tuple
\[(\Aa, \wedge, \delb , \Delta:=\del^\omega)\]
is a differential graded $\BV$-algebra.
The total differential
\[
\mathbb{D} =  \delb + \del^\omega \xi \in \End(\Aa)[[\xi]]
\]
 satisfies
\[
\mathbb{D} = e^{\Lambda \xi} \delb e^{ -\Lambda \xi}.
\]
More generally, if $(M,\omega)$ is an almost symplectic with real polarization $(\Ll,\Ll')$ and $\Ll$ is integrable then
\[(\Aa, \wedge, \delb , \Delta_1 :=\del^\omega + [\La, [\del,L]]^\omega, \Delta_2 = - [\del,L]^\omega )
\]
is a commutative $BV_\infty$-algebra.
\end{theo}

\begin{proof} The last statement is simply one bi-graded component of Theorem \ref{thm;AS}, and it implies the first in the case $\del \omega = 0$.
\end{proof}

In order to obtain canonical homotopy hypercommutative structures we shall choose a metric 
that makes the Lagrangian sub-bundles orthogonal. In this case,
the Laplacian $\Delta_\delb:=\delb\delb^*+\delb^*\delb$ is elliptic 
for compact manifolds and so Hodge theory is available. We have:

\begin{lemm}[Symplectic-Lagrangian deformation retract]\label{defcanonicalL}
Let $(M,\omega)$ be a compact symplectic manifold with a real polarization $(\Ll,\Ll')$ such that $\Ll$ is integrable.
Given a metric making $\Ll$ and $\Ll'$ orthogonal, there is 
a canonical deformation retract $(\iota,\rho,h)$
for the de Rham complex $(\Aa,\delb)$ where the maps 
 $\iota$ and $\rho$ are given by taking $\delb$-harmonic representatives, and projecting to $\delb$-harmonic forms, respectively, and $h:=-\delb^* G_\delb$.
\end{lemm}

Using the above deformation retract, we obtain:

\begin{cor}\label{hycomLagrangian}
Let $(M,\omega)$ be a compact symplectic manifold with a real polarization $(\Ll,\Ll')$, such that $\Ll$ is integrable.
Given a metric making $\Ll$ and $\Ll'$ orthogonal,  there is a canonical homotopy hypercommutative algebra structure extending the commutative product. 
\end{cor}

\section{Chevalley-Eilenberg algebras and examples}\label{SectionExamples}
In order to compute examples, in this section we will resort to the theory of nilmanifolds.
Let $G$ be a nilpotent Lie group with Lie algebra $\g$.
If $\g=\g_\mathbb{Q}\otimes\mathbb{R}$
has a rational structure, then by a well-known result of Malcev there exists a
discrete subgroup $\Gamma$ such that the quotient $M=\Gamma \backslash G$ is a compact manifold.
Recall that the Chevalley-Eilenberg algebra associated to $\g$ is given by
\[\Aa^*_{\g}:=\bigoplus_{k\geq 0} \Lambda^k(\g^\vee).\]
The
differential on $\g^\vee$ is defined as the negative of the dual of the Lie bracket
and is extended uniquely to a derivation of $\Aa^*_{\g}$.
The commutative dg-algebra $\Aa^*_{\g}$ is
isomorphic to the algebra of left-invariant forms on $M$. This gives an inclusion
of commutative dg-algebras $\Aa^*_{\g}\hookrightarrow \Aa^*_{\dR}(M)$ which by Nomizu's Theorem, induces an isomorphism in cohomology
$H^*(\g)\cong H^*(M,\R)$.
In particular $\Aa^*_{\g}$ encodes the real homotopy type of the nilmanifold $M$ (and is actually a minimal model of $M$ in the sense of Sullivan).
Additional geometric structures on $M$, such a Poisson, symplectic, almost-Hermitian or Hermitian structures,
may be defined at the Lie algebra level, allowing for finite-dimensional linear algebra computations of the homotopical structures that arise.

\subsection{Canonical deformation retracts}
Consider a real inner product $\langle-,-\rangle$ on $\g$. This defines a
\textit{Hodge star operator}
\[\star:\Aa_\g^k\to \Aa_\g^{n-k}\quad\text{ by }\quad\alpha\wedge \star\beta=\langle \alpha,\beta\rangle\ \Omega,\]
where $\Omega$ is a volume form determined by the metric 
and $n$ is the dimension of $\g$. 
The space of $d$-harmonic forms is then defined by
\[\Hh^k(\g):=\Ker(d)\cap \Ker(d^*)\cap \Aa_\g^k,\] where $d^*:\Aa_\g^k\to \Aa_\g^{k-1}$ is given by $d^*=- \star d\star$.
Under the assumption that $H^n(\g)\cong \RR$ (which is the algebraic analogue of having a closed manifold),
$d^*$ is the formal adjoint to $d$. In this case, we obtain an orthogonal direct sum Hodge decomposition
\[\Aa_\g^k\cong \Hh^k(\g)\oplus d(\Aa^{k-1}_\g)\oplus d^*(\Aa^{k+1}_\g)\]
and there are isomorphisms $\Hh^k(\g)\cong H^k(\g)$.

Taking the standard inner product we obtain a canonical deformation retract $(\iota,\rho,h)$ for the Chevalley-Eilenberg algebra of $\g$,
where $\iota$ and $\rho$ are defined by projecting to harmonic forms and $h=-d^*G_d$, where $G_d$ denotes the Green operator.
This is just a Chevalley-Eilenberg version of Lemma \ref{defcanonicalR} for compact Riemannian manifolds.

\subsection{Poisson structures}
A \textit{Poisson structure} on $\g$ is a bi-vector $\pi\in \Lambda^2 \g$ such that $[\pi,\pi]=0$.
This makes $\Aa_\g$ into a differential graded $\BV$-algebra with $\Delta=[i_\pi,d]$,
where $i_\pi:\Aa_\g^*\to \Aa_g^{*-2}$ is the contraction operator by $\pi$.
Moreover, the equation
\[e^{i_\pi\xi} de^{-i_\pi \xi}=d+\Delta \xi\]
is satisfied. By Theorem \ref{hycomprop} and using the canonical deformation retract defined above, we get:
\begin{cor}
Let $\g$ be a real Lie algebra of dimension $n$ such that $H^n(\g)\cong \RR$. Then $\Aa_\g$ carries a canonical hypercommutative algebra structure extending the commutative product.
\end{cor}

In the case when $M=G/\Gamma$ is a nilmanifold with $Lie(G)=\g$, any Poisson structure on $\g$ gives a (left-invariant) Poisson structure on $M$.
The same formula $\Delta=[i_\pi,d]$ makes the de Rham algebra of $M$ into a $\BV$-algebra and  the inclusion $\Aa^*_{\g}\hookrightarrow \Aa^*_{\dR}(M)$ is a quasi-isomorphism of $\BV$-algebras. 

\begin{exam}[Poisson structures on the Kodaira-Thurston manifold]\label{KTPoisson}
The Kodaira-Thurston manifold is a compact nilmanifold of dimension 4
whose defining nilpotent Lie algebra is given by
$\g=\RR\{X,Y,Z,T\}$ with the only non-trivial Lie bracket given by $[X,Y]=-Z$.
This is the simplest example of a manifold admitting both symplectic and complex structures, but no Kähler structure. Its algebra of left-invariant forms is given by
$\Aa_\g\cong \Lambda(x,y,z,t)$ with $dz=xy$, and so its cohomology ring is given by
\[H^1\cong\RR\{[x],[y],[t]\}, H^2\cong\RR\{[xz],[xt],[yz],[yt]\},H^3\cong \RR\{[xyz],[xzt],[yzt]\},
H^4\cong \RR\{[xyzt]\}.
\]

Any Poisson structure on $\g$ has the form
\[\pi=a X\wedge Z+bX\wedge T+c Y\wedge Z+eY\wedge T+fZ\wedge T.\]
Computing $\Delta=[i_\pi,d]$,  we obtain
 \[\Delta(xzt)=b(xy), \Delta(yzt)=e(xy)\text{ and }\Delta(zt)=ex-by,\]
 and $\Delta=0$ for the remaining cochain complex generators.
The maps $\rho$ and $\iota$ of the canonical deformation retract described above are obvious. 
The homotopy is given by $h(xy)=-z$, $h(xyt)=-zt$ and is zero otherwise.
We may now compute $\varphi_1=h\Delta-\iota\rho \Delta h$. Its only non-trivial values are 
\[\varphi_1(xyt)=ex-by,\, \varphi_1(xzt)=-bz\text{ and }\varphi_1(yzt)=-ez.\]
Using the expression for $m_3$ of Corollary \ref{cor;cBVdefretrTOhh}, we get
\[m_3(x,y,t)=\varphi_1(xyt)=ex-b y,\]
which induces a non-trivial strict hypercommutative operation in cohomology whenever $e$ or $b$ are non-zero.
\end{exam}

\subsection{Almost Hermitian structures} 
A linear (almost) complex structure on $\g$ is given by an endomorphism
 $J$ of $\g$ such that $J\circ J=-\mathrm{Id}$. The complexified Chevalley-Eilenberg algebra 
then becomes a bigraded algebra $\Aa^*_{\g_\CC}=\bigoplus \Aa^{p,q}_{\g}$ with
\[
\Aa^{p,q}_{\g}:=\Lambda^p\left((\mathfrak{g^\vee_\CC})^{1,0}\right) \wedge \Lambda^q \left((\mathfrak{g^\vee_\CC})^{0,1} \right).
\]
A real inner product $\langle-,-\rangle$ on $\g$ satisfying the compatibility condition
$\langle J X, JY\rangle = \langle X,Y\rangle$
 extends to a Hermitian inner product on $\g_\CC=\g\otimes\CC$.
The pointwise version of Theorem \ref{mainAH} gives:

\begin{cor}\label{pointwiseAH}
Let $(\g,J,\langle-,-\rangle)$ be a real Lie algebra of dimension $2n$ together with an almost complex structure $J$ and a compatible metric. Then
$\Aa_\g$ carries the structure of a commutative $\BV_\infty$-algebra extending the commutative product.
Moreover, if $H^{2n}(\g)\cong \RR$ then $H^*(\g)$ carries a canonical homotopy hypercommutative algebra.
\end{cor}

For a general almost complex Lie algebra, the exterior differential on $\Aa_\g^*$
 decomposes into four components $d=\mub +\delb +\del +\mu$.
The integrability of $J$ is measured by the Nuijenhuis tensor
\[N_J(X,Y):=[X,Y]+J[X,JY]+J[X,JY]-[JX,JY]\]
and $N_J\equiv 0$ if and only if $d=\delb+\del$.
In this case, the pointwise version of Theorem \ref{mainHermitian} gives a Dolbeault version of the above corollary:

\begin{cor}\label{pointwiseAHC}
Let $(\g,J,\langle-,-\rangle)$ be a real Lie algebra of dimension $2n$ together with an integrable almost complex structure $J$ and a compatible metric. Then
$(\Aa_{\g_\CC},\delb)$ carries the structure of a commutative $\BV_\infty$-algebra extending the commutative product.
Moreover, if $H^{n,n}(\g_\CC)\cong \CC$ then $H^*_\delb(\g_\CC)$ carries a canonical homotopy hypercommutative algebra.
\end{cor}

\begin{exam}[Complex structure on the Kodaira-Thurston manifold]
Let us now consider the complex structure $J$ on $\g=\RR\{X,Y,Z,T\}$ defined by
$J(X) = Y$ and $J(Z) = -T$.
Let $A=X-iJX$ and $B=Z-iJZ$. The only-non trivial Lie bracket of $\g$ is $[X,Y]=-Z$, and we get  
\[[A,\ov{A}]=-i(B+\ov{B}).\]
Therefore we may write $\Aa^*_{\g_\CC}\cong \Lambda(a,b,\ov{a},\ov{b})$ with  
$db=\delb b=ia\ov{a}$ and $d\ov{b}=\del \ov{b}=ia\ov{a}$.
The Dolbeault cohomology reads:
\[H_\delb^{*,*}\cong 
\arraycolsep=4pt\def\arraystretch{1.4}
 \begin{array}{|c|c|c|c|}
 \hline
[\ov a \ov b]&[b \ov a \ov b ]&[ab\ov a\ov b] \\
  \hline
 [\ov a], [\ov b]&[a\ov b], [b \ov a ]&[ab\ov a ], [ab\ov b] \\
   \hline
 1&[a]&[ab]\\ 
 \hline
\end{array}
\]
Consider the standard inner product on $\mathfrak{g}$. This makes the obvious representatives of the Dolbeault cohomology classes harmonic.
The fundamental $(1,1)$-form determined by this metric is 
\[
\omega = \frac{i}{2} \left( a \bar a + b \bar b \right).
\]

The canonical deformation retract gives 
\[h(a \bar a) = -ib\text{ and }h(a \bar a \bar b) = -i b \bar b\]
and zero otherwise. 
We may now compute 
\[\varphi_1=h\Delta-\iota\rho \Delta h,\text{ where }\Delta=[\Lambda,\delb]=-i(\del^*+\tau^*).\]
We obtain 
\[\varphi_1(a\ov a)=-2i,\, \varphi_1(ab\ov a)=4ib,\, \varphi_1(a\ov a\ov b)=-4i\ov b\text{ and }
\varphi_1(ab\ov a\ov b)=-8ib\ov b.\]
Using the expression for $m_3$ of Corollary \ref{cor;cBVdefretrTOhh}, we get
\[m_3(a,\ov a,\ov b)=\varphi_1(a\ov a\ov b)-\varphi_1(a\ov a)\ov b=-2i\ov b,\]
which induces a non-trivial strict hypercommutative operation in Dolbeault cohomology.
\end{exam}

 \subsection{Symplectic-Lagrangian structures}
A non-degenerate Poisson structure on a real Lie algebra $\g$ of even dimension gives a symplectic form $\omega$ on its Chevalley-Eilenberg algebra $\Aa_\g$. A \textit{Lagrangian polarization} for $(\g,\omega)$ is just a decomposition $\mathfrak{g}=\Ll\oplus \Ll'$ such that $\dim \Ll =\dim \Ll'$ and $\omega|_{\Ll}=\omega|_{\Ll'}=0$. The subspace $\Ll$ is \textit{integrable} if it is closed under the Lie bracket of $\g$.

 \begin{exam}[Kodaira-Thurston manifold: symplectic-Lagrangian case] 
 Taking $a=e=1$ in Example \ref{KTPoisson} and zero otherwise, we obtain left-invariant symplectic
form $\omega = yt+xz$ on the Kodaira-Thurston manifold, which has an integrable Lagrangian subspace $L=\mathbb{R}\{X,T\}$, in bi-degree $(0,1)$. In this case, we have $d = \delb$ since the only nonzero differential is $dz = \delb z = xy \in \cA^{1,1}$. Thus, the operation $m_3$ obtained in Example \ref{KTPoisson} is also a nontrivial hypercommutative operation in $\delb$-cohomology of the symplectic-Lagrangian case.
 \end{exam}

 The above example is somewhat trivial, as the hypercommutative algebra defined on the Lagrangian structure does not refine the one defined on the underlying Poisson structure. A more interesting example is the following:

 \begin{exam}[Iwasawa manifold: symplectic-Lagrangian case] \label{ex;Iwas-L}  
 
 The Iwasawa manifold is a 6-dimensional compact nilmanifold with real nilpotent Lie algebra
$\mathfrak{g}= \mathbb{R}\{X_1,X_2,X_3,X_4,X_5,X_6\}$,  defined by the non-trivial Lie brackets
\[[X_1,X_3]=X_5,\, [X_2,X_4]=-X_5,\, [X_1,X_4]=X_6\text{ and }[X_2,X_3]=X_6.\]
The symplectic form $\omega = x_1x_6+x_2x_5+x_3x_4$ with $\Ll=\mathbb{R}\{X_1,X_3,X_5\}$ is Lagrangian and integrable, while $\Ll'=\mathbb{R}\{X_2,X_4,X_6\}$ is Lagrangian but not integrable.

With this bigrading, $x_1,x_3,x_5 \in \cA^{0,1}$, $x_2,x_4,x_6 \in \cA^{1,0}$ and the only nonzero differentials are 
\begin{align*}
dx_5 &= x_3x_1 -x_4x_2, \quad \textrm{i.e.} \quad \mu x_5 = - x_4x_2 \quad \textrm{and}\quad \delb x_5 =x_3x_1, \\
dx_6 &= x_3x_2 + x_4x_1, \quad \textrm{i.e.} \quad \quad \delb x_6 =x_3x_2+x_4x_1.
\end{align*}
Note that $\mub=0$ and therefore $\delb^2=0$, so Proposition \ref{thm;s-L} applies.

Clearly, $H^1_{\delb} \cong \mathbb{R}\{[x_1],[x_2],[x_3],[x_4]\}$, and for classes $a,b,c \in H^1_{\delb}$ 
\[
m_3(a,b,c) = \varphi_1(abc) \quad \textrm{where} \quad \varphi_1 = h \Delta_1 -  \iota \rho \Delta_1 h,
\]
where $\Delta_1 = \del^\omega = [\La,\delb]$. For the examples below, only the second summand
$ \iota \rho \Delta_1 h$ will be non-zero. We compute
\[
m_3(x_2,x_3,x_4)=\varphi_1(x_2x_3x_4)=-\iota\rho \Delta_1h(x_2x_3x_4)=
\iota\rho\Delta_1(x_4x_6)=  \iota\rho(x_2)= x_2.
\]
\end{exam}

\bibliographystyle{alpha}
\bibliography{biblio}

\end{document}